\documentclass[a4paper,12pt,intlimits,oneside]{amsart}
\usepackage{graphicx}

\usepackage{enumerate}
\usepackage{amsfonts}
\usepackage{amsmath,amsthm,amssymb}

\usepackage{mathabx}

\newcommand{\comment}[1]{}

\newtheorem{theorem}{Theorem}
\newtheorem{definition}[theorem]{Definition}
\newtheorem{proposition}[theorem]{Proposition}
\newtheorem{lemma}[theorem]{Lemma}
\newtheorem{corollary}[theorem]{Corollary}
\newtheorem{problem}[theorem]{Problem}
\newtheorem{example}[theorem]{Example}
\newtheorem{remark}[theorem]{Remark}

\newtheorem{condition}[theorem]{Condition}

\newcommand{\inv}{^{-1}}

\newcommand\supp{\qopname\relax o{supp}}

\newcommand\con{\qopname\relax o{con}}

\newcommand\n{\nu}

\newcommand\f{\varphi}

\newcommand\Ne{\mathcal{N}_e}

\newcommand{\NN}{\mathbb N}
\newcommand{\ZZ}{\mathbb Z}
\newcommand{\RR}{\mathbb R}
\newcommand{\CC}{\mathbb C}
\newcommand{\TT}{\mathbb T}
\newcommand{\QQ}{\mathbb Q}

\newcommand\CF{\mathcal{C}}

\newcommand\MM{\mathcal{M}}
\newcommand\KK{\mathcal{K}}
\newcommand\Om{\Omega}

\newcounter{reb}
\newcounter{rev}
\newcounter{rem}
\newcounter{rec}
\newcommand{\beql}[1]{\begin{equation}\label{#1}}

\newcommand{\eeq}{\end{equation}}

\newcommand{\ve}{\varepsilon}

\newcommand{\FF}{{\mathcal F}}

\begin{document}

\title[Pointwise maximal value problem in groups]{The point value maximization problem for \\ positive definite functions supported in \\ a given subset of a locally compact group}


\author{S\'andor Krenedits and Szil\'ard Gy. R\'ev\'esz}\thanks{Supported in part by the Hungarian National Foundation for Scientific Research, Project \#'s
K-100461, 
NK-104183,
K-109789. 
}

\date{\today}

\address{
\newline \indent A. R\'enyi Institute of Mathematics
\newline \indent  Hungarian Academy of Sciences
\newline \indent Budapest, Re\'altanoda utca 13--15.
\newline \indent 1053 HUNGARY
}
\email{krenedits@t-online.hu, revesz.szilard@renyi.mta.hu}

\begin{abstract}
The century old extremal problem, solved by Carath\'eodory and
Fej\'er, concerns a nonnegative trigonometric polynomial
$T(t)=a_0+\sum_{k=1}^n a_k\cos (2\pi k t) + b_k \sin(2\pi kt)
\geq 0$, normalized by $a_0=1$, and the quantity to be
maximized is the coefficient $a_1$ of $\cos (2\pi t)$.
Carath\'eodory and Fej\'er found that for any given degree $n$
the maximum is $2 \cos(\frac{\pi}{n+2})$.

In the complex exponential form, the coefficient sequence
$(c_k)\subset \CC$ will be supported in $[-n,n]$ and normalized
by $c_0=1$. Reformulating, nonnegativity of $T$ translates to
positive definiteness of the sequence $(c_k)$, and the extremal
problem becomes a maximization problem for the value at $1$ of
a normalized positive definite function $c:\ZZ\to \CC$,
supported in $[-n,n]$.

Boas and Katz, Arestov, Berdysheva and Berens, Kolountzakis and
R\'ev\'esz and recently Krenedits and R\'ev\'esz investigated
the problem in increasing generality, reaching analogous
results for all locally compact Abelian groups. We prove an
extension to all the known results in not necessarily
commutative locally compact groups.
\end{abstract}

\maketitle

\vskip1em \noindent{\small \textbf{Mathematics Subject
Classification (2000):} Primary 43A35, 43A70. \\ Secondary 42A05, 42A82. \\[1em]
}

\section{Introduction}\label{sec:intro}

In this work we consider locally compact groups. Our aim is to extend a number of results stemming from the classical trigonometric polynomial extremal problem of Carath\'eodory and Fej\'er, to the generality of locally compact, (but in general not Abelian) groups. Apart from the original works of Carath\'eodory \cite{Cara} and Fej\'er \cite{Fej}, our results will also contain corresponding results of \cite{ABB}, \cite{BK}, \cite{kolountzakis:pointwise}, and \cite{KR}.

A century ago Carath\'eodory and Fej\'er originally addressed
the question of maximizing the coefficient of $\cos x$ in a
$2\pi$-periodic trigonometric polynomial $T(x)\geq 0$ of degree
at most $N$ and normalized to have constant term 1. Our
starting point is the observation, that the classical problem
of Carath\'eodory and Fej\'er (and many others) can be
formulated in the following fairly general way even in not
necessarily Abelian (topological) groups.

\begin{problem}\label{p:ptwgeneral} Let $\Omega\subset G$ be a given set in the group $G$ and let $z\in \Omega$ be fixed. Consider a \emph{positive definite function} $f: G\to \CC$ (or $\to \RR$), normalized to have $f(e)=1$ for the unit element $e \in G$, and \emph{vanishing outside of $\Omega$}. How large can then $|f(z)|$ be?
\end{problem}

Actually, the Carath\'eodory-Fej\'er extremal problem -- as is briefly discussed below -- corresponds to the special case when $G=\ZZ$ and $\Om=[-N,N]$. Therefore, a proper generalization to groups is already obtained assuming also the following restriction.

\begin{condition}\label{c:CF} Assume that there is a natural number $N\in \NN$ satisfying
\begin{equation}\label{CFcondition}
\langle z \rangle \cap \Omega = \{z^n~:~ -N\leq n \leq N \}, \qquad \textrm{where} \quad \langle z \rangle := \{ z^n~:~n\in \ZZ\}.
\end{equation}
\end{condition}

The analogous problem of maximizing $\int_\Omega f$ under similar hypothesis was recently well investigated by several authors under the name of "Tur\'an's extremal problem"--which terminology originated from the paper of Stechkin \cite{stechkin:periodic}--although later it turned out that the problem was already considered well before Tur\'an, see the detailed survey \cite{LCATuran}. The problem in our focus, in turn, was also investigated on various classical groups (the Euclidean space, $\ZZ^d$ and $\TT^d$ being the most general ones) and was also termed by some as "the pointwise Tur\'an problem", but the paper \cite{kolountzakis:pointwise} traced it back to Boas and Kac \cite{BK} in the 1940's and even to the work of Carath\'eodory \cite{Cara} and Fej\'er \cite{Fej} \cite[I, page 869]{Fgesamm} as early as in the 1910's.

Let us admit right here that although \cite{kolountzakis:pointwise} extended the problem even to sets $\Omega$ in classical groups \emph{not satisfying} Condition \ref{c:CF}, in the generality of non-Abelian groups we cannot describe the problem without any extra assumption yet. Still, the unconditional extension can be handled in locally compact \emph{Abelian} groups, to which we refer the reader to \cite{KR}.

On the other hand here we consider locally compact groups
lacking any assumption about commutativity. Correspondingly, we
write the group operation the usual multiplicative way, and
investigate Problem \ref{p:ptwgeneral} either under the
Carath\'eodory-Fej\'er Condition, i.e. Condition \ref{c:CF},
or, in Sections \ref{sec:mainresults} and \ref{sec:mainproofs},
under another assumption defined later as \emph{roundness}, see
Definition \ref{def:round}. As we will explain in the sequel,
this latter extension is still general enough to cover all the
known extensions, including the unconditional results of
\cite{KR} for LCA groups.

To the best of our knowledge, this is the first attempt to deal with such extremal problems -- including other extremal problems of the "Tur\'an type", as mentioned above -- in the wider generality of \emph{not} necessarily Abelian locally compact groups. Yet, the extension is very natural, because all ingredients, in particular also positive definiteness, is naturally defined on all groups, not just on Abelian groups. It would be interesting to decide if these results extend to arbitrary $\Om$ and $z\in \Om$, without any additional conditions used e.g. in this paper.

We termed Problem \ref{p:ptwgeneral} -- at least under Condition \ref{c:CF} -- the \emph{Carath\'eodory-Fej\'er type extremal problem on $G$ for $z$ and $\Om$}. Since Carath\'eodory and Fej\'er worked on their extremal problem well before the notion of positive definiteness was introduced at all, this may require some explanation. Hewitt-Ross \cite[p.325]{HewittRossII} gives a detailed account of how the development of the notion of positive definiteness was ignited by such extremal problems and in particular by the work of Carath\'eodory himself.

Basically, the explanation is that nonnegativity of a trigonometric polynomial $T(x)=\sum_k c_k e^{2\pi i k x}$ can be equivalently spelled out as \emph{positive definiteness} of the coefficient sequence $(c_k)$ on $\ZZ$, an observation first made by Toeplitz, see \cite[p. 325]{HewittRossII} and \cite{Toeplitz}. Also, the information that the degree of the polynomial $T$ is at most $N$ means that $\supp c_k \subset [-N,N]$, and with $z:=1\in \ZZ$ this is easily seen to match the condition formulated in \eqref{CFcondition} with $\Om:=[-N,N]$. For more details about this interpretation we refer the reader to \cite{KR}.

For further use we also introduce the extremal problems
\begin{align}\label{eq:CForigdef}
\MM(\Om)& :=\sup \{ a(1) ~:~ a:[1,N]\to \RR,~ N\in \NN,~ a(n)=0 ~ (\forall n\notin \Om),
\\ &\qquad\qquad \qquad\qquad T(t):=1+\sum_{n=1}^N a(n) \cos(2\pi nt)\geq 0 ~ (\forall t\in \TT)\}, \notag
\end{align}
which is called in \cite{kolountzakis:pointwise} the \emph{Carath\'eodory-Fej\'er type trigonometric polynomial problem} and
\begin{align}\label{CFfinitemdef}
\MM_m(\Om)& :=\sup \{ a(1) ~:~ a:\ZZ_m\to \RR,~ a(0)=1, ~a(n)=0 ~ (\forall n\notin \Om),
\\ &\qquad\qquad \qquad\qquad T\left(\frac{r}{m}\right):=\sum_{n \mod m} a(n) \cos\left(\frac{2\pi n r}{m}\right) \geq 0 ~ (\forall r \mod m) \}. \notag
\end{align}
which is termed in \cite{kolountzakis:pointwise} as the
\emph{Discretized Carath\'eodory-Fej\'er type extremal
problem}. 

\begin{remark}\label{r:MmandM} Obviously we have $\MM_m(\Om) \ge \MM(\Om)$, because the restriction on the admissible class of positive definite functions to be taken into account is lighter for the discrete problem: we only require $T\left(\frac{r}{m}\right)\geq 0~(0\leq r <m)$, while for $\MM(\Om)$ the restriction is $T(t)\geq 0 ~(\forall t\in \TT)$.
\end{remark}

Note that Problem \ref{p:ptwgeneral} may have various
interpretations depending on how we define the exact class of
positive definite functions what we consider, what topology we
use on $G$, if any, and how we formulate the restrictions with
respect to the function $f$ "living" in $\Omega$ only, or
regarding "nicety" of $f$. In case of the analogous "Tur\'an
problem" when one maximizes the integral $\int_G f d\mu_G$
rather than just a fixed point value $|f(z)|$, consideration of
various classes are more delicate, see \cite[Theorem
1]{kolountzakis:groups}.

In the Carath\'eodory-Fej\'er extremal problem, however, the
general approach on LCA groups was found to be largely
indifferent to these issues in \cite{KR}. That will be the case
also for not necessarily commutative groups, so following
\cite{KR} let us restrict to the two extremal cases. That is,
denoting positive definiteness of a function $f$ by writing
$f\gg 0$, we define here only
\begin{eqnarray}
\label{Fjustpd} \FF_G^{\#}(\Om) &:=& \{f:G\to \CC~~:~~ f\gg 0, ~f(e)=1,~f(x)=0~\forall x\notin\Om\,\}~,
\\ \label{Fcontcompact} \FF^c_G(\Om) &:=& \{f:G\to \CC~~:~~ f\gg 0, ~f(e)=1,~ f\in C(G),~~\supp f
\Subset\Om~ \} \, .
\end{eqnarray}
Once again, note that the first formulation is absolutely free of any topological or measurability structure of the group $G$. On the other hand, equipping $G$ with the discrete topology the latter gives a formulation close to the former but with restricting $f$ to have finite support.

Let us call the attention to the fact that for $f \gg 0$ the support $\supp f$ is always a \emph{symmetric} closed subset of $G$ (as it follows immediately from \eqref{eq:ftildedef} and \eqref{eq:fequalsftilde} below). Therefore, given a subset $\Omega \subset G$, we can equivalently change it to the symmetric subset $\Omega \cap \Omega^{-1}$ in all these formulations. Further, either $e\in \Omega$, or the problem becomes trivial with $f$ not satisfying the normalization condition $f(e)=1$ and thus the respective function classes becoming $\FF_G^{\#}(\Om)=\FF^c_G(\Om)=\emptyset$. Consequently, we will assume throughout that the fundamental sets $\Omega \subset G$ are chosen to be symmetric subsets containing the unit element $e$.

The respective "Carath\'eodory-Fej\'er constants" are then
\begin{equation}\label{CFconstants}
\CF_G^{\#}(\Om,z)~ := \sup \bigg\{|f(z)|\,:~ f \in \FF_G^{\#}(\Om) \bigg\},
\quad \CF^c_G(\Om,z)~ := \sup \bigg\{|f(z)|\,:~ f \in \FF^c_G(\Om) \bigg\}.
\end{equation}
Note that for $G$ finite, the conditions if $f$ is continuous or supp$f$ is compact become automatically satisfied, whence there is no need to distinguish between these classes, and we can just write $\FF_G(\Om)(:=\FF^c_G(\Om)=\FF^{\#}_G(\Om))$ and $\CF_G(\Om,z)(:=\CF^c_G(\Om,z)=\CF^{\#}_G(\Om,z))$.

In view of \eqref{eq:posdefbdd} below, giving that for $f\gg 0$ $\|f\|_\infty=f(e)$, the {\em trivial estimate} or \emph{trivial (upper) bound} for the Carath\'eodory-Fej\'er constants is simply $\CF^c_G (\Om,z)\le  \CF_G^{\#}(\Om,z) \le 1$. Note that also the lower estimation $\CF_G^{\#}(\Om,z) \ge \CF^c_G (\Om,z) \ge 1/2$ can be worked out analogously to
\cite[Proposition 3.2]{KR}.

By the above general definition, for $G=\ZZ$ and $G=\ZZ_m:=\ZZ/m\ZZ$ the Carath\'eodory-Fej\'er constants \eqref{CFconstants} with $z:=1$ (and denoting by $H$ the fundamental set $\Om$ in this case) are
\begin{align}\label{eq:CFOmZ}
\CF^{\#}(H)& :=\CF^{\#}_\ZZ(H,1):=\sup \{ |\f(1)|~:~ \f\in \FF_\ZZ^{\#}(H) \}
\notag\\& :=\sup \{ |\f(1)|~~:~~ \f: \ZZ\to \CC, ~ \f\gg 0,  ~\f(0)=1,~ \supp \f \subset H \},
\notag \\ \CF^c(H)& :=\CF^c_\ZZ(H,1):=\sup \{ |\f(1)|~:~ \f\in \FF^c_\ZZ(H) \}
\\& :=\sup \{ |\f(1)|~~:~~ \f: \ZZ\to \CC, ~ \f\gg 0, ~\f(0)=1,~ \supp \f \subset H, ~ \# \supp \f <\infty  \}, \notag \\
\CF_m(H)& :=\CF^{\#}_{\ZZ_m}(H,1)=\CF^c_{\ZZ_m}(H,1):=\sup \{ |\f(1)|~:~ \f\in \FF_{\ZZ_m}(H) \}
\notag \\& :=\sup \{ |\f(1)|~~:~~ \f: \ZZ_m\to \CC, ~ \f\gg 0, ~\f(0)=1,~ \supp \f \subset H \}.\notag
\end{align}

Also the issue whether we consider complex valued functions or real valued functions, occurs naturally. As is discussed in \cite{KR}, for the Carath\'eodory-Fej\'er type problem the choice of function classes simplifies compared to the "Tur\'an problem", while the issue of considering real- or complex valued functions becomes less simple and in fact it splits in some cases while it remains equivalent for others.

So we put for any group, (and so in particular for $G=\ZZ$ and $G=\ZZ_m$)
\begin{equation}\label{eq:FGROm}
\FF_G^{\#\RR}(\Om):=\{\f: G\to \RR~:~\f\in \FF_G^{\#}(\Om)\},\qquad
\FF_G^{c\RR}(\Om):=\{\f: G\to \RR~:~\f\in \FF^c_G(\Om)\},
\end{equation}
and, as before, for finite groups we again simplify putting
$\FF_G^{\RR}(H):=\FF_G^{c\RR}(H)=\FF_G^{\#\RR}(H)$.
Then with this we can write
\begin{align}\label{CFKKconstants}
\KK^{\#}_G(\Om,z):=\sup_{\f\in\FF_G^{\#\RR}(\Om)}|\f(z)|,
&\quad \KK_G^c(\Om,z):=\sup_{\f \in \FF_G^{c\RR}(\Om)} |\f(z)|,
\\
\KK^{\#}(H):=\KK^{\#}_{\ZZ}(H,1), \quad
\KK^{c}(H):=\KK^{c}_{\ZZ}(H,1), &\quad
\KK_m(H):=\KK_{\ZZ_m}(H,1):=\sup_{\f\in \FF_{\ZZ_m}^{\RR}(H)} |\f(1)|.\notag
\end{align}

\begin{proposition}\label{p:realpd} We have for arbitrary symmetric $0\in H\subset \ZZ$ the following.
\begin{enumerate}
\item[{\it (i)}] $\MM(H\cap \NN)=2 \KK(H)$ and for all $m\in \NN$ $\MM_m(H\cap |0,m/2])=2 \KK_m(H)$.
\item[{\it (ii)}] We have $\KK^c(H)=\KK^{\#}(H)=\CF^c(H)=\CF^{\#}(H)$ (which can thus be denoted by $\CF(H)$ from now on).
\item[{\it (iii)}] For all $m\in \NN$, $\cos(\pi/m) \CF_m(H) \leq \KK_m(H) \leq \CF_m(H)$.
\item[{\it (iv)}][\textbf{Ruzsa}] If $4 \leq m\in \NN$, then in general {\it (iii)} is the best possible estimate with both inequalities being attained for some symmetric subset $H\subset \ZZ_m$.
\item[{\it (v)}] If $m=2,3$, then for any admissible $H$ we must have $H=\ZZ_m$ and thus $\f(x)\equiv 1$ shows $\CF_m(H)=\KK_m(H)=1$.
\end{enumerate}
\end{proposition}
\begin{proof}
This is a combination of \cite[Propositions 3.1]{KR} and
\cite[Propositions 3.5]{KR}, where in \emph{(iv)} an oral
communication of I. Z. Ruzsa is used, too.
\end{proof}

\section{Previous results on Carath\'eodory-Fej\'er type extremal problems}\label{sec:ptTuranintro}

As mentioned above, the development started with the extremal
problem of Carath\'eodory and Fej\'er, originally formulated
for positive trigonometric polynomials, see \cite{Cara, Fej} or \cite[vol. {\bf 1}, p.869]{Fgesamm}.

\begin{theorem}[{\bf Carath\'eodory and Fej\'er}]\label{th:CF}
If $T(t):=1+\sum_{n=1}^N a(n) \cos(2\pi nt)\geq 0 ~ (\forall t\in \TT)$, then $|a_1|\leq 2 \cos \left( \frac{\pi}{N+2}\right)$, and the bound is sharp. In other words, $\MM([1,N])=2 \cos \left( \frac{\pi}{N+2}\right)$.
\end{theorem}

Boas and Katz \cite{BK} used the result of Carath\'eodory and Fej\'er to prove the following.
\begin{theorem}[{\bf Boas-Katz}]\label{th:BK} Let $\Omega \subset \RR^d$ be a convex, symmetric, open, bounded set, i.e. one which generates a corresponding norm $\|\cdot\|=\|\cdot\|_\Om$ on $\RR^d$. Consider the Carath\'eodory-Fej\'er extremal problem Problem \ref{p:ptwgeneral}. Then with $\lceil x \rceil $ denoting upper integer part of $x$, we have
$$
\KK^{\#}_{\RR^d}(\Om,z)=\cos\left(\frac{\pi}{\lceil 1/\|z\| \rceil+1}\right).
$$
\end{theorem}
Observe that here $N< 1/\|z\|\leq N+1$ means by convexity and the definition of the norm that $\langle z \rangle \cap \Om = \{ nz~:~ n\in [-N,N] \}$, that is Condition \ref{c:CF}. As then also $\lceil 1/\|z\| \rceil=N+1$, by Proposition \ref{p:realpd} (i) the result is indeed the Euclidean space version of the Carath\'eodory-Fej\'er result on $\ZZ$. In view of Proposition \ref{p:realpd} (ii) this can be complemented by the equalities $\KK^{\#}_{\RR^d}(\Om,z)=\KK^c_{\RR^d}(\Om,z)= \CF^{\#}_{\RR^d}(\Om,z)=\CF^c_{\RR^d}(\Om,z)=
\cos\left(\frac{\pi}{\lceil 1/\|z\| \rceil+1}\right)$.

In particular, the result means that for an interval $(-h,h)\subset \RR$ and $z\in(-h,h)$ we have $\KK^{\#}_{\RR}((-h,h),z)=\KK^c_{\RR}((-h,h),z) =\CF^{\#}_{\RR}((-h,h),z)=\CF_{\RR}((-h,h),z)= \cos\left(\frac{\pi}{\lceil h/z \rceil+1}\right)$.

It is more difficult to deal with the case of the torus. The first results in this respect were obtained by Arestov, Berdysheva and Berens \cite{ABB}. A key observation is that for a symmetric, open set $\Om\subset (-1/2,1/2)^d$, we can consider it both as subset of $\RR^d$ and of $\TT^d$, and for the latter $\langle z \rangle \cap \Om $ always contains the one in the previous case. On the other hand if the order of $z\in\TT^d$ is a finite number $m$ -- which happens exactly when $z\in \QQ^d$ -- then we also have the effect mentioned in Remark \ref{r:MmandM}. Finally, the trigonometric versions and the exponential versions can be exchanged with each other, according to Proposition \ref{p:realpd} (i), so that finally $\KK_{\RR^d}^c(\Om,z)\leq \KK_{\TT^d}^c(\Om,z)$ follows, c.f. \cite[Proposition 6.1]{kolountzakis:pointwise}. As a corollary, improving upon the estimate of \cite[Theorem 3 (14)]{ABB}, it was also found that for a symmetric, open convex set $\Om\subset (-1/2,1/2)^d$, $\KK_{\TT^d}^c(\Om,z)\geq \cos \left(\frac{\pi}{\lceil 1/\|z\| \rceil+1}\right)$, see \cite[Corollary 6.2]{kolountzakis:pointwise}.

As mentioned above, it is difficult to compute exact extremal values if the order of $z$ is finite, say $o(z)=m$. More work was done on this for the case when $\Om=(-1/2,1/2)^d$, so in a sense only the "boundary" (only a smaller dimension subset) is removed from $\TT^d$. In this case it is known \cite[Theorem 7.3]{kolountzakis:pointwise} that for $z\not \in \QQ^d$ or $z \in \QQ^d$ but $z=(\frac{p_1}{q_1},\dots,\frac{p_d}{q_d})$ (in simplest terms) containing in the denominator exactly the same power $2^s$ of $2$ in each coordinates we have  $\KK_{\TT^d}^c(\Om,z)=1$, while for other $z\in \Om$ we have $\KK_{\TT^d}^c(\Om,z)=\frac12 (1+\cos(\frac{2\pi}{m}))$, where $m:=[q_1,\dots,q_d]$ is the least common multiple of the denominators in the coordinates of $z$. (Note that in this case we have that some of the denominators must be even, hence also $m$ is even and $m/2\in \NN$.) The key to the calculation is the computation in $\ZZ_m$ of the extremal value $\KK_{m}(\ZZ_m\setminus\{m/2\})=\frac12 \left(1+\cos(\frac{2\pi}{m})\right)$.

For a general interval $\Om=(-h,h)\subset \TT$ and the special value $z:=p/q$ with $p/q<h\leq (p+1)/q$ see the paper \cite{Ivanov}. (The result is quoted and explained also in \cite[Theorem 5.2]{KR}.)

Later developments surpassed the assumption of convexity or
restrictions of the kind of Condition \ref{c:CF}. Kolountzakis
and R\'ev\'esz \cite{kolountzakis:pointwise} already considered
open symmetric subsets $\Om\subset \RR^d$ or $\TT^d$, and
recently Krenedits and R\'ev\'esz \cite{KR} extended the
results to locally compact Abelian groups as well.

Perhaps the main result of \cite{kolountzakis:pointwise} is the understanding that the above point-value extremal problems in the more complicated groups $\RR^d$ and $\TT^d$ are in fact equivalent to the above trigonometric polynomial extremal problems on $\ZZ$ or $\ZZ_m$. Until that work the equivalence remained unclear in spite of the fact that, e.g., Boas and Kac found ways to deduce the solution of the trigonometric extremal problem \eqref{eq:CForigdef} from their results on Problem \ref{p:ptwgeneral}. The recent results of \cite{KR} capitalized on this observation when proving analogous equivalences even in LCA groups.

\begin{theorem}[{\bf Krenedits and R\'ev\'esz}]\label{th:KrenciLCA}
Let $G$ be any locally compact Abelian group and let
$\Om\subset G$ be an open (symmetric) neighborhood of $0$. Let
also $z\in \Om$ be any fixed point with $o(z)=\infty$, and
denote $H:=H(\Om,z):=\{k\in \ZZ~:~ kz\in \Om\}$. Then we have
\begin{align}\label{eq:KrenciLCA}
\CF^c_G(\Om,z) = \CF^{\#}_G(\Om,z) = \KK_G^{c}(\Om,z)=\KK_G^{\#}(\Om,z)=\CF(H).
\end{align}
\end{theorem}
Recall that in $\ZZ$ we have
$\CF(H):=\CF^c(H)=\CF^{\#}(H)=\KK^c(H)=\KK^{\#}(H)=:\KK(H)$.
\begin{theorem}[{\bf Krenedits and R\'ev\'esz}]\label{th:KrenciLCAfini}
Let $G$ be any locally compact Abelian group and let
$\Om\subset G$ be an open (symmetric) neighborhood of $0$. Let
also $z\in \Om$ be any fixed point with $o(z)=m<\infty$, and
denote $H_m:=H_m(\Om,z):=\{k\in \ZZ_m~:~ kz \in \Om\}$. Then we have
\begin{equation}\label{eq:KrenciLCAfini}
\CF^c_G(\Om,z) = \CF^{\#}_G(\Om,z) =\CF_m(H_m) \quad {\rm and}\quad \KK_G^{c}(\Om,z)=\KK_G^{\#}(\Om,z) =\KK_m(H_m).
\end{equation}
\end{theorem}

These results extended the corresponding Theorems 2.1 and 2.4
of \cite{kolountzakis:pointwise} to all LCA groups. Note that
these theorems are \emph{equivalence statements}, which greatly
decrease the complexity of the problems when reducing them to
given extremal problems on $\ZZ$ or $\ZZ_m$, yet they do not
necessarily \emph{solve} them in the sense of providing the
exact numerical value of the extremal quantity. Nevertheless,
it already follows that for $o(z)=\infty$ all the four versions
of the Carath\'eodory-Fej\'er extremal constants, listed in
\eqref{eq:KrenciLCA}, are equal -- and thus can be denoted by
$\CF_G(\Om)$ in general-- while for $o(z)<\infty$ at least the
two complex, resp. two real versions coincide and thus can be
denoted as $\CF_G(\Om)$ and $\KK_G(\Om)$, respectively. One of
the goals of the current work is to see this phenomenon in
non-commutative groups, too, under suitable conditions at
least, see later in particular in Corollaries
\ref{c:zinftyalleq} and \ref{c:ozfineq}. For the numerous
applications and exact computations or estimates of the
concrete values of these extremal constants see the original
papers and the references therein.

Also let us recall once again, that -- at least for the time
being -- we cannot obtain such fully general results for the
current non-commutative case: here we need to assume further
conditions, too, like e.g. Condition \ref{c:CF} or validity of
the condition of roundness, formulated in Definition \ref{def:round} appearing later. Still, our new results below will completely cover the theorems of Fej\'er and Carath\'eodory, Boas and Kac, Arestov, Berdysheva and Berens, Kolountzakis and R\'ev\'esz, and even of Krenedits and R\'ev\'esz, too.

\section{Formulation of the results under the Carath\'eodory-Fej\'er Condition}\label{s:result}

When the order of $z$ is infinite, the result will be an exact generalization of the Carath\'eodory-Fej\'er result.

\begin{theorem}\label{th:KrenciCF} Let $G$ be any locally compact topological group, with unit element $e$ and let $\Om\subset G$ be an open (symmetric) neighborhood of $e$. Let also $z\in \Om$ be any fixed point with $o(z)=\infty$, and assume that Condition \ref{c:CF} is satisfied with a certain $N$. Then we have
\begin{align}\label{eq:KrenciCF}
\CF^c_G(\Om,z) = \CF^{\#}_G(\Om,z) = \KK_G^{c}(\Om,z)=\KK_G^{\#}(\Om,z) = \cos\left(\frac{\pi}{N+2} \right).
\end{align}
\end{theorem}

So this also extends Theorem \ref{th:KrenciLCA} from LCA groups to not necessarily commutative locally compact groups in case \eqref{CFcondition} is satisfied.

If $z\in G$ is cyclic (has torsion), then we will show that Problem \ref{p:ptwgeneral} reduces to a well-defined discrete problem of the sort \eqref{CFfinitemdef}. Again, this is an extension of Theorem \ref{th:KrenciLCAfini} to not necessarily commutative locally compact groups in case \eqref{CFcondition} holds.

\begin{theorem}\label{th:KrenciCFfini} Let $G$ be any locally compact topological group, with unit element $e$ and let $\Om\subset G$ be an open (symmetric) neighborhood of $e$. Let also $z\in \Om$ be any fixed point with $o(z)=m<\infty$, and assume that Condition \ref{c:CF} is satisfied with some $N\leq m$. Then we have
\begin{equation}\label{eq:KrenciCFfini}
\CF^c_G(\Om,z) = \CF^{\#}_G(\Om,z) = \CF_m([-N,N]) \quad \textrm{and}\quad \KK_G^{c}(\Om,z)=\KK_G^{\#}(\Om,z) =\KK_m([-N,N]).
\end{equation}
\end{theorem}

\begin{remark}\label{r:CFresults} Note the slight difference between the two results. For $o(z)=\infty$ we get the exact value of $\cos\left(\frac{2\pi}{N+2} \right)$ of the extremal constant, while for $o(z)=m<\infty$ we only obtain an equivalence, but not the concrete value. This is a consequence of the fact that the numerical value of the Carath\'eodory-Fej\'er extremal problem is known, and also that for $o(z)=\infty$ the real and complex cases agree, while for the finite group analog neither the value(s) are known, nor equality of the complex and real settings is known to hold. It would be interesting to compute also this case, in particular as it depends now on two variables, $N$ and also $m$, which makes the calculation certainly nontrivial both for the real and also for the complex case. \\
Compare also to \cite{Ivanov}, where the "discrete Fej\'er type maximization problem" of maximizing the $\n¡^{\rm th}$ coefficient, i.e. computing $\KK_{\ZZ_m}([0,N],\nu)$ is formulated, and also computed for the special case of $\nu=N$, i.e. the last (and not the first) coefficient.
\end{remark}

The proofs of these results can be found in Section \ref{s:proof}.

\section{Preliminaries on positive definite functions on locally compact groups}\label{sec:posdefoverview}

Positive definite functions were introduced on $\ZZ$ by Toeplitz \cite{Toeplitz} in 1911 and on $\RR$ by Matthias in 1923 \cite{Mathias}. For general groups positive definite functions are defined by the property that
\begin{equation}\label{eq:posdefdfnd}
\forall n\in \NN, ~\forall x_1,\dots,x_n\in G, ~ \forall c_1,\dots,c_n\in \CC \qquad \qquad \sum_{j=1}^{n}\sum_{k=1}^{n} c_j \overline{c_k} f(x_jx_k^{-1}) \geq 0.
\end{equation}
In other words, positive definiteness of a real- or complex valued function $f$ on $G$ means that for all $n$ and all choices of $n$ group elements $x_1,\dots,x_n\in G$, the $n\times n$ square matrix $[f(x_jx_k^{-1})]_{j=1,\dots,n}^{k=1,\dots,n}$ is a positive (semi-)definite matrix. We will use the notation $f \gg 0$ for a short expression of the positive definiteness of a function $f:G\to \CC$ or $G\to \RR$.

Definition \eqref{eq:posdefdfnd} has some immediate consequences\footnote{These properties are basic and well-known, see e.g. \cite[(32.4) Theorem]{HewittRossII} or \cite[p. 84]{Folland}. We prove them just for being self-contained, as they are easy.}, the very first being that $f(e)\geq 0$ is nonnegative real (just take $n:=1$, $c_1:=1$ and $x:=e$).

For any function $f:G\to \CC$ the \emph{converse}, or \emph{reversed} function $\widetilde{f}$ (of $f$) is defined as
\begin{equation}\label{eq:ftildedef}
\widetilde{f}(x):=\overline{f(x^{-1})}.
\end{equation}
E.g. for the characteristic function $\chi_A$ of a set $A$ we have $\widetilde{\chi_A}=\chi_{A^{-1}}$ (where, as usual, $A^{-1}:=\{a^{-1}~:~a\in A\}$), because $x^{-1}\in A$ if and only if $x\in A^{-1}$.

Now let $f:G\to \CC$. Then in case $f$ is positive definite we necessarily have
\begin{equation}\label{eq:fequalsftilde}
f=\widetilde{f}.
\end{equation}

Indeed, take in the defining formula \eqref{eq:posdefdfnd} of positive definiteness $x_1:=e$, $x_2:=x$ and $c_1:=c_2:=1$ and also $c_1:=1$ and $c_2:=i$: then we get both $0\leq 2 f(e)+f(x)+f(x^{-1})$ entailing that $f(x)+f(x^{-1})$ is real, and also that $0\leq 2f(e)+if(x)-if(x^{-1})$ entailing that also $if(x)-if(x^{-1})$ is real. However, for the two complex numbers $v:=f(x)$ and $w:=f(x^{-1})$ one has both $v+w\in \RR$ and $i(v-w) \in \RR$ if and only if  $v=\overline{w}$.


Next observe that for any positive definite function $f:G\to \CC$ and any given point $z\in G$
\begin{equation}\label{eq:posdefbdd}
|f(z)|\leq f(e),
\end{equation}
and so in particular if $f(e)=0$ then we also have $f\equiv 0$. Indeed, let $z\in G$ be arbitrary: if $|f(z)|=0$, then we have nothing to prove, and if $|f(z)|\ne 0$, let $c_1:=1$, $c_2:=-\overline{f(z)}/|f(z)|$ and $x_1:=e$, $x_2:=z$ in \eqref{eq:posdefdfnd}; then recalling that according to \eqref{eq:fequalsftilde} $f(z^{-1})=\overline{f(z)}$, we obtain $0\leq 2f(e)+c_2 f(z)+\overline{c_2} f(z^{-1})=2f(e)-2|f(z)|$ and \eqref{eq:posdefbdd} follows.

Therefore, all positive definite functions are bounded and $\|f\|_\infty =f(e)$. That is an important property which makes the analysis easier: in particular, we immediately see that $|f(z)|$ cannot exceed 1 for $f\in \FF^{\#}_G(\Om)$, whence the above mentioned trivial upper bound.

If $u:G \to L(\mathcal H, \mathcal H)$ is a unitary representation of the group $G$ in the Hilbert space $\mathcal H$, and $v\in \mathcal H$ is a fixed vector, then it is easy to see that the function $x \to \langle u(x)v,v\rangle$ makes a positive definite function. In fact, all positive definite functions can be represented such a way, see \cite[(3.15) Proposition and (3.20) Theorem]{Folland} or better \cite[(32.3) Theorem (iii)]{HewittRossII}. As an easy consequence, any character $\gamma\in \widehat{G}$ of a LCA group $G$ is  positive definite.

We will make use of the following further constructions of positive definite functions.
\begin{lemma}\label{l:furtherposdeffunctions} Let $f,g :G\to \CC$ be arbitrary positive definite functions. Then we have
\begin{itemize}
\item[\it{(i)}]
If $H\leq G$ is a subgroup of $G$, and $h:=\chi_{H} g$, that is, $g|_{H}$ on $H$ and vanishing elsewhere, then also $h\gg 0$.
\item[\it{(ii)}] $\overline{f} \gg 0$, $f^{\star}(x):=f(x^{-1})\gg 0$ and $\Re f\gg 0$.
\item[\it{(iii)}]
If $\alpha, \beta>0$ are arbitrary positive constants, then $\alpha f + \beta g \gg 0$.
\item[\it{(iv)}]
For arbitrary $n\in \NN$, complex numbers $a_j\in \CC$ ($j=1,\dots,n$) and group elements $y_j\in G$ ($j=1,\dots,n$), the derived function $F(x):= \sum_{j=1}^n \sum_{k=1}^n \overline{a_j}a_k f(y_j^{-1}xy_k) \gg 0$.
\item[\it{(v)}] $fg \gg 0$.
\end{itemize}
\end{lemma}
\begin{proof} The parts {\it(i)--(iv)} are easy facts, which the reader will find no difficulty to check by a straightforward calculation, but we note that {\it(v)} is a nontrivial fact, which follows from the Schur Product Theorem: if the matrices $A, B\in \CC^{n\times n}$ are both positive definite, then so is their entrywise product $[a_{jk}b_{jk}]_{j=1,...,n}^{k=1,...,n}$, too. For this latter fact from linear algebra, see \cite[\S 85, Theorem 2]{Halmos} or \cite[(D.12) Lemma, Appendix D, pp. 683-684]{HewittRossII}.

All the statements {\it (i)-(v)} can be found in \cite[(32.8) (d)]{HewittRossII} and \cite[(32.9) Theorem]{HewittRossII}.
\end{proof}

As a converse to {\it (i)}, we can recall the following easy to prove observation of Hewitt, see \cite[(32.43) (a)]{HewittRossII}.

\begin{lemma}\label{l:extensions} Let $H\leq G$ be a subgroup of $G$ and let $h \gg 0$ on $H$. Furthermore, define $g :G\to \CC$ as its trivial extension:
\begin{equation}\label{eq:posdefextension}
g(x):=\begin{cases} h(x) \qquad &{\rm if}~ x\in H \\
0 &{\rm if}~ x\not\in H
\end{cases}.
\end{equation}
Then $g\gg 0$ on $G$.
\end{lemma}

The only problem with this is that it does not necessarily preserve analytic structure, like e.g. continuity, so in case we are involved with structural restrictions, such a statement in itself may not suffice. Actually, the crux of the proof of our results is to circumvent this technical obstacle and build a further function, still positive definite, and (essentially) preserving the values at $H$, but admitting some continuity (or at least integrability) restrictions, too. For a discussion of the (limited) possibility to use this ``trivial extension" in our constructions and proofs, in particular in case of cyclic or monothetic (sub-)groups, see \cite[\S 5]{KR}.

To fix notations, for locally compact (Hausdorff) topological groups $G$ we will consider their \emph{left-invariant Haar-measure} $\mu_G$ normalized the standard way so that for discrete groups all points have measure 1 and for non-discrete compact groups $\mu_G(G)=1$. Recall, that on $G$ there exist essentially unique left- and also right Haar measures, see e.g. \cite[Section 2.2]{Folland}.

For the fixed left Haar measure $\mu_G$ and for any fixed $g\in G$ the measure $H\to\mu_G(Hg)$ ($\forall H\subset G$ Borel measurable) provides another left Haar measure. The essential uniqueness of the left Haar measure thus gives that there is a function $\Delta:G\to (0,\infty)$ with $\mu_G(Hg)=\Delta(g)\mu_G(H)$ for all Borel measurable $H\subset G$ and any $g\in G$. Moreover, the so introduced \emph{modular function} $\Delta(g)$ is positive, continuous and multiplicative, i.e. a continuous group homomorphism from $G$ to $\RR_{+}$, see \cite[(15.11) Theorem]{HewittRossI}.

Correspondingly, we will consider the (left-) convolution of functions with respect to the left Haar-measure $\mu_G$, that is
\begin{equation}\label{eq:convolutiondef}
(f\star g) (x) := \int_G  f(y)g(y^{-1}x) d\mu_G(y)= \int_G  f(xz)g(z^{-1}) d\mu_G(z),
\end{equation}
defined for all functions $f,g \in L^1(\mu_G)$, see e.g. \cite[page 50]{Folland}. (The same convolution formula appears in \cite{HewittRossI} as a \emph{theorem}, namely (20.10) Theorem (i) on page 291, for Hewitt and Ross introduce convolution in a more general setting.)

As a concrete application, consider now the characteristic functions $\chi_A, \chi_B$ of two Borel measurable sets $A, B$ with finite left Haar measure. In this particular case definition \eqref{eq:convolutiondef} yields
\begin{equation}\label{eq:chistarchi}
\chi_A \star \chi_B (x) = \int_G \chi_A (y) \chi_B(y^{-1}x) d\mu_G(y) = \int_G \chi_A  \chi_{xB^{-1}} d\mu_G
= \mu_G (A \cap xB^{-1}).
\end{equation}
In particular, if $B=A^{-1}$, then $(\chi_A\star\chi_{A^{-1}})(x)=\mu_G(A\cap xA)$.

For more on convolution of measures and functions see e.g. \cite[pages 49-54]{Folland} and \cite[\S\S 19,20]{HewittRossI} (with respect to left Haar measure) or \cite[pages 180-]{Bachman} (however, the latter is w.r.t. right Haar measure). Note that convolution is in general \emph{not} commutative: more precisely, the convolution is commutative if and only if the locally compact group itself is commutative, see e.g. \cite[(20.24) Theorem]{HewittRossI}. We will need the next well-known assertion (which we will use only in the rather special case of characteristic functions, however).

\begin{lemma}\label{l:convolutionsquare} Let $f\in L^2(\mu_G)$ be arbitrary. Then the``convolution square" of $f$ exists, moreover, it is a continuous positive definite function, that is, $f\star \widetilde{f}\gg 0$ and
$f\star \widetilde{f} \in C(G)$.
\end{lemma}
\begin{proof}
This can be found for locally compact \emph{Abelian} groups in \cite[\S 1.4.2(a)]{rudin:groups}, and for the general case in e.g. \cite[Example 4, p. 220]{Bachman} or \cite[(32.43) (e)]{HewittRossII}. It also follows from a combination of \cite[(3.35) Proposition]{Folland} and
\cite[(3.16) Corollary]{Folland}.
\end{proof}

Although it is very useful when it holds, in general this statement cannot be reversed. Even for classical Abelian groups, it is a delicate question when a positive definite continuous function has a "convolution root" in the above sense. For a nice survey on the issue see e.g. \cite{EGR}. We will however be satisfied with a very special case, where this converse statement follows from a classical result of L. Fej\'er and F. Riesz, see \cite[Theorem 1.2.1]{Szego}, \cite{Fej}, or \cite[I, page 845]{Fgesamm}. For details about the below reformulation in the language of positive definite sequences see \cite[Lemma 2.2 (i)]{KR}.

\begin{lemma}\label{l:FR} Let $\psi:\ZZ\to \CC$ be a finitely supported positive definite sequence. Then there exists another sequence $\theta:\ZZ\to \CC$, also finitely supported, such that $\theta\star\widetilde{\theta}=\psi$. Moreover, if $\supp \psi \subset [-N,N]$, then we can take $\supp \theta \subset [0,N]$.
\end{lemma}

A slightly less strict analog of the existence of a convolution square-root also holds in $\ZZ_m$. Again, for the--simple and straightforward--proof see \cite[Lemma 2.2 (ii)]{KR}.

\begin{lemma}\label{l:FRfinite}  If $\psi:\ZZ_m\to\CC$, $\psi\gg 0$ on $\ZZ_m$,
then there exists $\theta:\ZZ_m\to \CC$ with $\theta\star\widetilde{\theta}=\psi$.
\end{lemma}

Note the slight loss of precision -- we cannot bound the support of $\theta$ in terms of a control of the support of $\psi$. This is natural, for the same finitely supported sequence can be positive definite on $\ZZ_m$ more easily than on $\ZZ$, in view of the fact that the equivalent restriction that its Fourier transform (the trigonometric polynomial $T$) satisfies $T(2\pi n/m)\geq 0$ ($\forall n\in \ZZ_m$) is met more easily than $T(t)\geq 0$ ($\forall t\in \TT$).

\section{Proofs of the Results under Condition \ref{c:CF}}\label{s:proof}

\begin{proof}[Proof of Theorem \ref{th:KrenciCF}]
Let us write throughout the proof $H:=[-N,N]$. According to Proposition \ref{p:realpd} (i), (ii) and the result of Carath\'eodory and Fej\'er i.e. Theorem \ref{th:CF}, we have $\CF^c(H)=\CF^{\#}(H)=\KK^c(H)=\KK^{\#}(H)=\cos\left(\frac{\pi}{N+2} \right)$, hence it suffices to prove $\CF^c_G(\Om,z)=\CF^{\#}_G(\Om,z)=\CF^c(H)$ and $\KK^c_G(\Om,z)=\KK^{\#}_G(\Om,z)=\KK^c(H)$.

The complex and real cases are exactly similar, so we work out only one, say the complex case. As obviously we have $\CF^c_G(\Om,z)\leq \CF^{\#}_G(\Om,z)$, we are to prove only two inequalities, the first being that $\CF^{\#}_G(\Om,z) \leq \CF^{c}(H)$.

Since now $H\subset \ZZ$ is finite, there can be no difference between $\CF^{c}(H)$ and $\CF^{\#}(H)$, so it suffices to show $\CF^{\#}_G(\Om,z) \leq \CF^{\#}(H)$. Let now take any $f\in \FF^{\#}_G(\Om)$. Consider the subgroup $Z:=\langle z \rangle \leq G$. Noting that according to Lemma \ref{l:furtherposdeffunctions} {\it{(i)}}, $g:=f|_Z \gg 0$ on $Z$, we have defined a function $g \in \FF^{\#}_Z((\Om\cap Z))$. Finally, let us remark that the natural isomorphism $\eta:\ZZ\to Z$, which maps according to $\eta(k):=z^k$, carries over $g$, defined on $Z\le G$, to a function $\psi:=g\circ \eta: \ZZ\to \CC$, which is therefore positive definite on $\ZZ$, has normalized value $\psi(0)=g(e)=f(e)=1$, and $\supp \psi \subset H$ for $\supp g \subset (\supp f \cap Z) \subset (\Om \cap Z) = \{ z^k:|k|\leq N\}=\eta(H)$ in view of \eqref{CFcondition}. That is, $\psi \in \FF^{\#}_\ZZ(H)$.

From here we read that $|f(z)|=|g(z)|=|\psi(1)| \leq \sup \left\{ |\f(1)| ~:~  \f \in \FF^{\#}_\ZZ(H) \right\} =\CF^{\#}(H)$. Taking $\sup_{f\in \FF^{\#}_G(\Om)}$ on the left hand side concludes the proof of the first part.

\smallskip

It remains to show $\CF^c(H) \leq \CF^{c}_G(\Om,z)$.

So let us consider an arbitrary function $\psi\in \FF^c(H)$. Since $\psi \gg 0$, its Fourier transform, the trigonometric polynomial $T(t):=\widecheck{\psi}(t)=\sum_{n=-N}^N \psi(n) e^{2\pi i n t}$ is nonnegative, and we can invoke the classical theorem of Fej\'er and Riesz to represent it as a square $T(t)=|P(t)|^2$. Using the above described preliminaries precisely, we can apply Lemma \ref{l:FR}. Thus we find another sequence $\theta:[0,N]\to \CC$ such that $\psi(n)=\sum_{k=0}^N \theta(k)\overline{\theta(k-n)}=(\theta\star\widetilde{\theta}) (n)$.

Let us consider some compact neighborhood $U$ of $e$. Observe that for any given value $-N\leq n \leq N$ we have at most $N+1$ different pairs $0\leq k,j \leq N$ with $j-k=n$ (while for $|n|>N$ we have none). If $j-k=n$, then $z^{j}Uz^{-k}$ is the continuous image of the compact set $U$ under the continuous homeomorphism $x\to z^{j}xz^{-k}$. These continuous homeomorphisms take $e$ to $z^n$, hence $z^{j}Uz^{-k}$ is a compact neighborhood of $z^n$, and so is the set $U_n:= \bigcup \{z^{j}Uz^{-k}~:~ 0\leq j,k\leq N, k-j=n\}$. Note that by Condition \ref{c:CF}, all these $z^n\in \Om$.

So obviously, if $U\Subset \Om$ is chosen small enough, we can ensure that also $U_n\Subset \Om$, hence also $\bigcup_{n=-N}^N U_n \Subset \Om$.
Also, the finitely many $z^n$ with $n\in H=[-N,N]$ are all different (for $o(z)=\infty$), whence for $U$ chosen small enough, we can ensure $z^n\not\in U$ for all $0<|n|\leq N$.

Next, let $f$ be any positive definite function, supported compactly on the above chosen $U \Subset \Om$. Then by Lemma \ref{l:furtherposdeffunctions} \emph{(iv)}, also the function $F(x):=\sum_{j=0}^N \sum_{k=0}^N \theta(k)\overline{\theta(j)} f(z^{-j}xz^{k})$ is positive definite, moreover, we obviously have $\supp F \Subset \bigcup_{j,k=0}^N z^{j}Uz^{-k}=\bigcup_{n=-N}^N U_n$. Therefore, by the above choice of $U$, we have $\supp F \Subset \Om$, too.

Now let us compute $F(e)$ and $F(z)$. By the above choice of $U$ and $F$ we have that
$$
F(e)=\sum_{j,k=0}^N \theta(k)\overline{\theta(j)} f(z^{-j+k})= \sum_{j=0}^N \theta(j)\overline{\theta(j)} f(e) + \sum_{j=0}^N \sum_{0\leq k\leq N \atop k\ne j} \theta(k)\overline{\theta(j)} f(z^{-j+k})=f(e),
$$
because the second sum vanishes in view of $z^n\not \in \supp f \subset U$ ($n\ne 0$) by the choice of $U$, while $\sum_{j=0}^N |\theta(j)|^2=\psi(0)=1$. Similarly, we compute
$$
F(z)
=\sum_{j,k=0}^N \theta(k)\overline{\theta(j)} f(z^{-j+1+k})
= \sum_{j=1}^{N} \theta(j-1)\overline{\theta(j)} f(e) + \sum_{j=0}^{N} \sum_{0\leq k\leq N \atop k\ne j-1} \theta(k)\overline{\theta(j)} f(z^{-j+1+k})=\overline{\psi(1)} f(e),
$$
because again the second sum vanishes and $\sum_{j=1}^{N} \theta(j-1)\overline{\theta(j)}=(\theta\star\widetilde{\theta})(-1)=\psi(-1)=\overline{\psi(1)}$.

So \emph{if there exists a function} $f\gg 0$, continuous and with $\supp f \Subset U$, then we also have another continuous $F\gg 0$ with $\supp F \Subset \bigcup_{n=-N}^N U_n \Subset \Om$ having the properties that $F(e)=\psi(0)f(e)$ and $F(z)=\overline{\psi(1)}f(e)$. Of course, $f\equiv 0$ would qualify, but we want also $f(e)\ne 0$, which, in view of \eqref{eq:posdefbdd}, is equivalent to $f\not \equiv 0$. So with any such not identically vanishing function $f$ we can construct the above $F$, and even normalize it to $F_0:=\frac{1}{f(e)} F$, to obtain $F_0\in \FF^c(\Om)$ with $F_0(z)=\overline{\psi(1)}$. Then this entails that
$$
|\psi(1)|=|F_0(z)| \leq \sup \{|\phi(z)|~:~ \phi \in \FF^c_G(\Om)\}=\CF^c_G(\Om,z).
$$
Taking supremum on the left hand side over all $\psi\in \FF^c_\ZZ(H)$ thus leads to $\CF^c(H)\leq \CF^c_G(\Om,z)$, provided we can find a continuous, not identically zero $f\gg 0$, supported in $U$.

It only remains to construct  a continuous $f\gg 0$ with $\supp f \Subset U$ for the given compact neighborhood $U$ of $e$.

As $U$ is a neighborhood of $e$ and group multiplication is continuous, there exists an open neighborhood $V$ of $e$ such that $VV^{-1} \subset U$. Let now $W\Subset V$ be chosen such that $\mu_G(W)\ne 0$ (to which it suffices e.g. that $W$ is a closure of an open, nonvoid set). Then the characteristic function $g:=\chi_W$ of $W$ is certainly in $L^{\infty}(G)\subset L^2(G)$, and we can apply the above Lemma \ref{l:convolutionsquare} to get that $f:=g\star\widetilde{g} \gg 0$ and $f\in C(G)$. Note that $\widetilde{g}=\chi_{W^{-1}}$ and $f(x)=(\chi_{W}\star \chi_{W^{-1}})(x)=\mu_G(W \cap xW)$, see \eqref{eq:chistarchi}. So in particular $f(e)=\mu_G(W)\ne 0$, and also $\supp f \subset \overline{W W^{-1}} =W W^{-1} \subset V V^{-1} \subset U$, as needed.

Plugging this $f$ in the above construction thus completes the proof of $\CF^c(H)\leq \CF^c_G(\Om,z)$, whence the theorem.
\end{proof}

\begin{proof}[Proof of Theorem \ref{th:KrenciCFfini}] The direct proof is rather similar to the preceding one, once we carefully change all references from $\ZZ$ to $\ZZ_m$, $\CF(H)$ to $\CF_m(H)$ and $\FF_\ZZ(\Om,z)$ to $\FF_{\ZZ_m}(\Om,z)$, and note that $Z:=\langle z \rangle$ is now only a finite subgroup with $Z\cong \ZZ_m$, so the natural isomorphism $\eta(k):=z^k$ acts between $\ZZ_m$ and $Z$ now.

Note that here we have $H(\Om,z)=\{k\in\ZZ_m~:~z^k\in \Om\}$ in view of $z^m=e$ and Condition \ref{c:CF}, whence $\CF_m(H)=\CF_m([-N,N])$ and $\KK_m(H)=\KK_m([-N,N])$.

However, instead of describing such a slightly varied proof, we can as well postpone the proof until the later Theorem \ref{th:KrenciStrongfini} is proven. As that result is more general, from there only the specialization of $H(\Om,z)=[-N,N]$ is needed to infer Theorem \ref{th:KrenciCFfini}. In this regard it is of importance that the condition of roundness, defined and used later in the corresponding theorem for $o(z)=\infty$, is not needed in the fully general, unconditional Theorem \ref{th:KrenciStrongfini} for the finite order case. The reason for this formal difference between the corresponding results will also be explained in the next section.
\end{proof}

\section{A new condition on topological groups and on their elements}\label{sec:roundcondi}

Now instead of Condition \ref{c:CF} we introduce a different
one. Recall that $e\in G$ stands for the unit element of the
group $G$, and the left invariant Haar measure is denoted by
$\mu_G$. Also note that the \emph{symmetric difference} of two
subsets of $G$ are defined by $A\triangle B:= (A\setminus B)
\cup (B\setminus A)$ ($A,B\subset G$).

\begin{definition}\label{def:round}
We say that $z\in G$ is a \emph{round} element of the locally
compact group $G$, if for all open neighborhood $U$ of the unit
$e$ there exists another open neighborhood $V\subset U$ of $e$
such that $\mu_G(zVz^{-1}\triangle V) =0$.

Furthermore, the \emph{group} $G$ itself is called \emph{round}, if all
elements $z\in G$ are round according to the above.
\end{definition}

\begin{remark}\label{r:round} The seemingly one-sided definition does not depend on using left- or right-Haar measure. Indeed, if e.g. $z\in G$ is \emph{left-round}, and $U$ is any open neighborhood of $G$, then consider $U^{-1}\subset G$, which is another open neighborhood of $e\in G$, and take another neighborhood $V\subset U^{-1}$ of $e$ with the postulated property that $zVz^{-1}$ agrees $V$ apart from a $\mu_G$-null set. Taking inverses, it means that for the right Haar-measure $\mu_G^*(A):=\mu_G(A^{-1})$ and $W:=V^{-1}\subset (U^{-1})^{-1}=U$ we have $0=\mu_G(zVz^{-1}\triangle V)=\mu_G^*((zVz^{-1}\triangle V)^{-1})=\mu_G^*(zWz^{-1}\triangle W)$, proving that $z$ is right-round, too.
\end{remark}

\begin{remark}\label{r:amenable} Note that the condition is similar, but not equivalent to, the F{\o}lner Condition, which is an equivalent condition for amenability of the topological group $G$.

The F{\o}lner condition postulates that for every finite or compact set $F\subset G$, and for any $\ve>0$, there exists a measurable set $U$ of positive finite Haar measure, such that $\mu_G(U\triangle gU)<\ve \mu_G(U)$ holds for all $g\in F$. On the one hand, this seemingly requires more, as the set $U$ must be \emph{uniformly} given for all $g\in F$, on the other hand it does not require vanishing, but only relative smallness, of the measure of the symmetric difference. This seemingly little deviation causes that in fact $\mu_G(U\triangle gU)$ is typically large, however, $\mu_G(U)$ is even much larger. An easy example is any compact interval in $\RR$, with $U$ a much larger interval. Such large sets $U$ also exist in all Abelian groups in view of a well-known result of LCA groups, see e.g. 2.6.7 Theorem on page 52 of \cite{rudin:groups}. (This also means, as is well-known, that all LCA groups are amenable.)

However, the major difference between the two notions is the one-sided multiplication in amenability, and conjugation in roundedness. Inner automorphisms (i.e. conjugations) may distort sets, they even may change the Haar measure (depending on the value of the modular function), however, the definition of roundedness means that some (arbitrarily small) neighborhoods stay invariant. While amenability is related at the first place to the possibility of introducing an invariant mean on the group, roundedness is kind of a topological substitution to commutativity. If elements themselves do not always commute, roundedness formulates the property that elements and some arbitrarily small neighborhoods (a topological basis system) can still be interchanged. Therefore, one may perhaps call this property "topological commutativity".
\end{remark}

\begin{proposition}\label{p:invariance} A point $z\in G$ is round if and only if for any neighborhood $U$ of $e$ there exists a nonzero, continuous, positive definite function $f\gg 0$ such that $\supp f \Subset U$ and $f$ is invariant under conjugation by $z$: $f(z^{-1}xz)=f(x)$ for all $x\in G$.
\end{proposition}
\begin{proof}
Let us suppose, that there exists such a function $f$. Recall that $f(0)=\|f\|_\infty =\max\{|f(x)|: x\in G\}$, which must be positive by condition. Let $V:=\{x: |f(x)|>0\}$. Than $V$ is an open set, and $e\in V$, as $|f|$ attains the maximum at $e$. Furthermore for all $x \in V$, $f(zxz^{-1})=f(x)\neq 0$, whence also $zxz^{-1} \in V$. It follows, that
$(zVz^{-1}\triangle V)=\emptyset$, $\mu_G(zVz^{-1}\triangle V)
=0$, so $z$ is a round element.

On the other hand, take a compact neighborhood $W$ of $e$ with
$WW^{-1} \subset U$, which is possible for $(x,y)\to xy^{-1}$
is continuous from $G\times G$ to $G$. Moreover, by continuity
$U_0:=WW^{-1}$ is also compact, hence $U_0 \Subset U$. Now we
use here the roundedness assumption on $z$: we select an open
neighborhood $V\subset W$ of $e$ with $z^{-1}Vz$ matching $V$
apart from a $\mu_G$-null set of their symmetric difference.
Then we define $f:=\chi_V\star \widetilde{\chi_V}=\chi_V\star
\chi_{V^{-1}}$, having value $f(x)=\mu_G(V\cap xV)$ according
to the calculations at \eqref{eq:chistarchi}. Note that as
before, by Lemma \ref{l:convolutionsquare} $f$ is a continuous,
compactly supported positive definite function, with support
$\supp f \Subset \overline{VV^{-1}}\Subset U_0\Subset U$.
\end{proof}

\begin{example}\label{ex:Abel} An $LCA$ group $G$ is round.

Indeed, commutativity implies $zVz^{-1}=V$ for any points $z\in
G$ and sets $V\subset G$.
\end{example}

\begin{example}\label{ex:Centrum}
In any locally compact group $G$, the elements of the
\emph{centrum} $C(G):=\{a\in G~: ag=ga  ~ \forall g\in G\}$ are
round elements.

Indeed, for any $U$ taking $V:=U$ is an open neighborhood of $e$ within $U$ and obviously $z\in C(G)$ has the property $zUz^{-1}=\{zuz^{-1}~:~ u\in U\}=U$.

Note that in Abelian groups $C(G)=G$, whence this is more
general than Example \ref{ex:Abel}.
\end{example}

\begin{example}\label{ex:OpenC}
If the centrum $C:=C(G)$ is open in $G$, then 
$G$ is round.

Indeed, for an open neighborhood $U$ of $e$ take $V:=U\cap
C\subset C$: then for any $z\in G$ we have $zVz^{-1}= \{
zvz^{-1}~:~v\in V\}=V$, for $z$ and $v$ commute for all $v\in
V\subset C$.

Again, in an LCA group this holds, as the centrum is the full
group $G$, whence open.
\end{example}

\begin{example}\label{ex:ozm} If $z\in G$ is a cyclic element (has torsion), then it is also round.

Indeed, let $o(z)=m<\infty$ and $e\in U\subset G$ open. Then also the conjugates $z^jUz^{-j}$ are open, and even the intersection $V:=\cap_{j=0}^{m-1} z^jUz^{-j}$ will be an open neighborhood of $e$. Moreover, $zVz^{-1} =\cap_{j=1}^{m} z^jUz^{-j}=V$, for $z^mUz^{-m}=eUe^{-1}=U$.
\end{example}

\begin{example}\label{ex:discrete} Let $G$ be any group, equipped with the discrete topology. Then $G$ is round.

Indeed, $V:=\{e\}$ works for all choices of $U$ and for all $z$.
\end{example}

\begin{example}\label{ex:normal}
Suppose $G$ has a topological basis $\{ V_\alpha ~:~
\alpha\in A\}$ at $e$ to generate the topology of $G$ and
consisting of normal subgroups. Then $G$ is round.

Indeed, then for any $U$ there is $V:=V_\alpha\subset U$, with $V\lhd G$ (i.e. a normal
subgroup of $G$), to which all conjugates $z^{-1}Vz=V$.

Note that this covers e.g. the case of an arbitrarily large
cardinality box- or topological product of finite, discrete, or
otherwise round locally compact groups.
\end{example}

Of course, it is also interesting to see examples of not round groups and elements. To find such $z\in G$ becomes easy once we note the next simple observation.

\begin{lemma}\label{l:roundmodular} Let $z\in G$ be a round element. Then we have for the modular function $\Delta(z)=1$.
\end{lemma}
\begin{proof} First, let us take any compact neighborhood $U$ of $e$, and to it some open $e \in V\subset U$, essentially invariant under the conjugation by $z$, as guaranteed by the roundedness of $z$. The left Haar-measure of the set $V$ is finite (as $U$ is compact), but positive (as $V$ is open). According to roundedness, $\mu_G(z^{-1}Vz\triangle V)=0$, thus we must have $\mu_G(V)=\mu_G(z^{-1}Vz)$, too. Using the left invariance of $\mu_G$ and the definition of the modular function we find $\mu_G(V)=\mu_G(z^{-1}Vz)=\mu_G(Vz)=\mu_G(V)\Delta(z)$, so using $\mu_G(V)\ne 0$ the assertion follows.
\end{proof}

As a result, any locally compact topological group $G$, which is not unimodular, is neither round. Moreover, any $z\in G$ with $\Delta (z)\ne 1$ is not a round element.

\begin{corollary}\label{c:unimodular} If a group $G$ is round, then it must be unimodular.
\end{corollary}

In other words, finding a not round group or element is easy-- just take elements with the modular function not assuming 1 at it.

Recall\footnote{All the following nice comments in this section, which connect our notion of roundness to other well-known notions in locally compact groups, we thank to Prof. Pierre-Emmanuel Caprace, who kindly provided us these in an e-mail on November 15, 2013.} that in a locally compact group $G$ an element $g \in G$ is called:
\begin{itemize}
\item \textbf{equicontinuous} if there is a basis of identity neighborhoods in $G$ that are all invariant under conjugation by $g$;
\item \textbf{distal} if for every $x \in G \setminus \{e\}$, the closure $\overline{\{g^n x g^{-n} \; | \; n \in \ZZ\}}$ does not contain $e$.
\end{itemize}

Note that (according to Ellis, who refers to L. Zippin, see the closing paragraph of \cite{Ellis}) distality was considered by Hilbert in an attempt to give a topological characterization of the concept of a rigid group of motions. The closely related notions of distal elements, contraction groups and equicontinuous elements provided tools for exploring the structure of locally compact topological groups, see e.g. \cite{BaumWill, CM, CRW, Jaw, RajaShah} and the references therein.

\begin{proposition}
An element $g \in G$ is round if and only if it is equicontinuous.
\end{proposition}

\begin{proof}
The `if' part is clear. Assume conversely that $g$ is round and let $W$ be an identity neighborhood. Then there is an identity neighborhood $U \subset W$ such that $U \cdot U\inv \subset W$. By the roundness of $g$, we can find an identity neighborhood $V \subset U$ such that $\mu(V \triangle gVg\inv) = 0$. It follows that the set $V' = \bigcap_{n \in \ZZ} g^n V g^{-n}$ has the same measure as $V$; in particular $\mu(V')>0$. By \cite[Corollary 20.17]{HewittRossII}, the set $V' \cdot (V')\inv$ is an identity neighborhood. By construction it is $g$-invariant, and contained in $W$. Thus $g$ is equicontinuous.
\end{proof}

Recall that for any group element $g \in G$ the \emph{contraction group of $G$} is defined as $\con(g):=\{ x\in G~:~ g^nxg^{-n}\to e ~(n\to \infty) \}$. Note that this set is indeed trivially subgroup of $G$, by continuity of the group multiplication and taking inverses.

Distal elements can be characterized in terms of their contraction groups, as follows.

\begin{proposition}
An element $g \in G$ is distal if and only if its contraction group $\con(g)$ is trivial.
\end{proposition}

\begin{proof}
See Corollary~4.13 in \cite{RajaShah}.
\end{proof}

In case the ambient group is totally disconnected, the following result, due to Baumgartner--Willis and Jaworski, implies that all three properties are equivalent.
\begin{proposition}
Let $G$ be a totally disconnected locally compact group. An element $g \in G$ is   equicontinuous if and only if   $\con(g) = 1$.
\end{proposition}

\begin{proof}
The `only if' part is clear, for by the property of the topology on $G$ we have for any $x\in G$ and identity neighborhood $U\in \Ne$ with $x\not\in U$, and then by the (total) roundness of $g$ there is another, $g$-invariant open identity neighborhood $e\in V\subset U \not\ni x$, so that for all $n\in \NN$ the conjugate $g^nxg^{-n}\not\in g^nUg^{-n}$, whence $\not\in g^nVg^{-n}=V$, and $e$ cannot be an accumulation point of the sequence $(g^nxg^{-n})_{n=1}^\infty$.

The converse follows from Proposition~3.16 in \cite{BaumWill} in case $G$ is metrizable. All the results from \cite{BaumWill} have however been extended to the general case by Jaworski \cite{Jaw}.
\end{proof}

\section{More general results on the Carath\'eodory-Fej\'er problem}\label{sec:mainresults}

\begin{theorem}\label{th:KrenciCFstrong} Let $G$ be any locally compact topological group, with unit element $e$ and let $\Om\subset G$ be an open (symmetric) neighborhood of $e$. Let also $z\in \Om$ be any torsion-free point (i.e. $o(z)=\infty$) which is \emph{round} in the sense of Definition \ref{def:round}. Then with $H:=H(\Om,z):=\{k\in \ZZ~:~ z^k \in \Om\}$ we have
\begin{align}\label{eq:KrenciCFstrong}
\CF^c_G(\Om,z) = \CF^{\#}_G(\Om,z) = \KK_G^{c}(\Om,z)=\KK_G^{\#}(\Om,z) =\CF(H).
\end{align}
\end{theorem}

Note that in view of Example \ref{ex:Abel} this result extends Theorem \ref{th:KrenciLCA}.

\begin{corollary}\label{c:zinftyalleq} For $G$ a locally compact group,
$\Om\subset G$ an open (symmetric) neighborhood of $e$, and
$z\in \Om$ any fixed \emph{round} point with $o(z)=\infty$, we
have
$\CF^c_G(\Om,z)=\KK_G^c(\Om,z)=\CF^\#_G(\Om,z)=\KK_G^\#(\Om,z)$,
the common value of which can thus be denoted simply by
$\CF_G(\Om,z)$.
\end{corollary}

If $z\in \Om$ is a cyclic (torsion) element, then the situation is described by the next result.

\begin{theorem}\label{th:KrenciStrongfini} Let $G$ be any locally compact topological group, with unit element $e$ and let $\Om\subset G$ be an open (symmetric) neighborhood of $e$. Let also $z\in \Om$ be any cyclic point with $o(z)=m<\infty$, and let $H_m:=H_m(\Om,z):=\{k\in \ZZ_m ~:~ z^k \in \Om\}$. Then we have
\begin{equation}\label{eq:KrenciStrongfini}
\CF^c_G(\Om,z) = \CF^{\#}_G(\Om,z) = \CF_m(H_m) \quad \text{\rm and}\quad \KK_G^{c}(\Om,z)=\KK_G^{\#}(\Om,z) =\KK_m(H_m).
\end{equation}
\end{theorem}

Note that here we did not assume any extra condition, neither
\eqref{CFcondition} nor roundedness of $z$. Still, we obtain
the full strength of the result, and thus an unconditional
extension of the result of Theorem \ref{th:KrenciLCAfini}.
Dropping the roundedness condition, assumed for torsion-free
$z$, is possible here in view of Example \ref{ex:ozm}.

\begin{corollary}\label{c:ozfineq} For $G$ a locally compact group,
$\Om\subset G$ an open (symmetric) neighborhood of $e$, and
$z\in \Om$ any fixed point with $o(z)<\infty$, we still have
$\CF^c_G(\Om,z)=\CF^\#_G(\Om,z)$ and
$\KK_G^c(\Om,z)=\KK_G^\#(\Om,z)$, the common value of which can
thus be denoted by $\CF_G(\Om,z)$ and $\KK_G(\Om,z)$,
respectively.
\end{corollary}

\section{Proofs of the general results of \S \ref{sec:mainresults}}\label{sec:mainproofs}

\begin{proof}[Proof of Theorem \ref{th:KrenciCFstrong}] According to Proposition \ref{p:realpd} (i), (ii) we have $\CF^c(H)=\CF^{\#}(H)=\KK^c(H)=\KK^{\#}(H)$, hence it suffices to prove $\CF^c_G(\Om,z)=\CF^{\#}_G(\Om,z)=\CF(H)$ and $\KK^c_G(\Om,z)=\KK^{\#}_G(\Om,z)=\KK(H)(=\CF(H))$.

The complex and real cases are exactly similar, so we work out only the complex case. Again, in view of the obvious $\CF^{c}(H)\le \CF^{\#}(H)$, it suffices to prove two inequalities, the first being that $\CF^{\#}_G(\Om,z) \leq \CF(H)=\CF^{\#}(H)$.

Let now take any $f\in \FF^{\#}_G(\Om)$. Consider the subgroup $Z:=\langle z \rangle \leq G$. By Lemma \ref{l:furtherposdeffunctions} {\it{(i)}}, $g:=f|_Z \gg 0$ on $Z$, so we have defined a function $g \in \FF^{\#}_Z(\Om\cap Z)$. Consider the natural isomorphism $\eta:\ZZ\to Z$, which maps according to $\eta(k):=z^k$. Then $\eta$ carries over $g$, defined on $Z\le G$, to a function $\psi:=g\circ \eta: \ZZ\to \CC$, which is therefore positive definite on $\ZZ$, has normalized value $\psi(0)=g(e)=f(e)=1$, and $\supp \psi \subset H$ for $\supp g \subset (\supp f \cap Z) \subset (\Om \cap Z) = \{ z^k~:~z^k \in \Om\}=\eta(H)$ by definition of $H$.

From here we read that $|f(z)|=|g(z)|=|\psi(1)| \leq \sup \left\{ |\f(1)| ~:~  \f \in \FF^{\#}_\ZZ(H) \right\} =\CF^{\#}(H)$. Taking $\sup_{f\in \FF^{\#}_G(\Om)}$ on the left hand side concludes the proof of the first part.

\smallskip

Note that in this part we did not use the condition of "roundedness", and we could circumvent any compactness (or finiteness on $Z$) requirement on the support of $f\in \FF^{\#}_G(\Om)$. This is a key issue to obtain finally also $\CF^{\#}_G(\Om,z)=\CF^{c}_G(\Om,z)$ as a consequence of the two inequalities proved.

\smallskip
Now we turn to the second inequality, i.e. we show
$\CF(H)(=\CF^c(H) )\leq \CF^{c}_G(\Om,z)$.

So let us consider an arbitrary function $\psi\in \FF^c(H)$.
Here it is of relevance, that we have the finiteness assumption
on $S:=\supp \psi$, which is thus included in $[-N,N]$ for some
$N\in \NN$. This finiteness guarantees not only that $\psi \gg
0$, but also that its Fourier transform consists of finitely
many terms and so is a trigonometric polynomial
$T(t):=\widecheck{\psi}(t)=\sum_{n=-N}^N \psi(n) e^{2\pi i n t}$.

The next idea here is that as $\psi\gg 0$, we have that $T$ is
nonnegative, and we can invoke the classical theorem of Fej\'er
and Riesz to represent it as a square $T(t)=|P(t)|^2$. Using
the above described preliminaries precisely, we can apply Lemma
\ref{l:FR} to $\psi\gg 0$. Thus we find another sequence $\theta:[0,N]\to \CC$
such that $\psi(n)=(\theta\star\widetilde{\theta})
(n)=\sum_{k=\max(n,0)}^{\min(N,N+n)}
\theta(k)\overline{\theta(k-n)}~(|n|\leq N)$.

Let us consider some compact neighborhood $U$ of $e$. Observe
that for any given value $-N\leq n \leq N$ we have at most $N+1$
different pairs $0\leq k,j \leq N$ with $k-j=n$ (while for
$|n|>N$ we have none). If $k-j=n$, then $z^{-j}Uz^k$ is the
continuous image of the compact set $U$ under the continuous
homeomorphism $x\to z^{-j}xz^k$. These continuous mappings take
$e$ to $z^n$, hence $z^{-j}Uz^k$ is a compact neighborhood of
$z^n$, and so are the sets
$$
U_n:= \bigcup \{z^{-j}Uz^k~:~ 0\leq j,k\leq N, k-j=n\}=\bigcup_{k=\max(n,0)}^{\min(N,N+n)} z^{n-k}Uz^k=z^n \bigcup_{k=\max(n,0)}^{\min(N,N+n)} z^{-k}Uz^k
$$
for all $n=-N,\dots,N$.

However, we cannot postulate here that all these $z^n\in \Om$, and even the less that $U_n\Subset \Om$. The only fact we have for sure is that \emph{whenever
$\psi(n)\ne 0$, the point $z^n$ belongs to $\Om$.} This is so
because $\psi(n)\ne 0$ implies that $n\in S:=\supp \psi \subset
H=H(\Om,z)$, which means by construction that $z^n\in \Omega$.
The difficulty here is that even if $n$ can occur as $k-j$ with
$k,j\in \supp \theta$, it may still happen that
$\psi(n)=(\theta\star\widetilde{\theta}) (n)=\sum_{0\leq
k,j\leq N,~k-j=n}\theta(k)\overline{\theta(j)} =0$, and so we
have no guarantee that $z^n\in \Omega$.

So before proceeding, let us clarify what conditions we assume
when choosing the initial set $U$. We require the following to
hold.
\begin{itemize}
\item[(i)] $U^*:=\bigcup_{k=0}^N z^{-k}Uz^k \subset \Om$;
\item[(ii)] $U_n^*:=z^n U^* \subset \Om$ whenever $n\in S$;
\item[(iii)] If $0<|n|\leq N$, then $z^n\not\in U^*$, i.e.
    $e\not\in z^{-n}U^*=U_{-n}^*$; still equivalently,
    $z^n\not\in U^*_k$ unless $n=k$.
\end{itemize}

Choosing $U$ a sufficiently small neighborhood of $e\in\Om$,
together with $U$ also the homeomorphic images $z^{-k}Uz^k$ of
$U$ will stay within $\Om$, hence (i) is easy to satisfy.

Now if $n\in S\subset H$, then also $z^n\in \Om$, i.e. $e\in
z^{-n} \Om$, and for fixed $n$ a suitably small choice of $U$
can ascertain that the $N+1$ homomorphic images $z^{-k}Uz^k$ of
$U$ all stay within $z^{-n}\Om$ -- that is, $z^n U^* \subset
\Om$. Considering the necessary restrictions for the finitely
many elements $n\in S$ it is also possible to simultaneously
satisfy (ii) for all $n\in S$.

Finally, consider (iii). Let us take $\Om^*:=\Om\setminus
\{z^n~:~ 0<|n|\leq N\}$, which is a finite modification of
$\Om$, and is still a symmetric open neighborhood of $e$ (for
$z^n\ne e$ if $0<|n|\le N$). So arguing as in (i), we can thus
ascertain that $U^*\subset \Om^*$, whence $z^n\not\in U^*$, and (iii) holds.

\smallskip

Next, let $f$ be any continuous and positive definite function, supported compactly in the above chosen $U \Subset \Om$, and invariant under the effect of conjugation by $z$. By Proposition \ref{p:invariance} such an $f$ exists, for $z\in G$ was assumed to be a round element of $G$. Therefore, according to Lemma \ref{l:furtherposdeffunctions} \emph{(iv)}, also the function $F(x):=\sum_{j=0}^N \sum_{k=0}^N \theta(k)\overline{\theta(j)} f(z^{j}xz^{-k})$ is continuous and positive definite, moreover, we obviously have $\supp F \Subset \bigcup_{j,k=0}^N z^{-j}Uz^k=\bigcup_{n=-N}^N U_n\subset \bigcup_{n=-N}^N U_n^*$, hence $F$ is compactly supported.

We, however, need more, as we also need $\supp F \Subset \Om$. Therefore we compute the support of $F$ more precisely, showing that it in fact lies within $\Om$. To this end let us calculate $F(x)$ for a general point $x$, using the invariance of $f$ with respect to conjugation by $z$. We compute
\begin{align}\label{eq:Fcompu}
F(x)&=\sum_{j,k=0}^N \theta(k)\overline{\theta(j)} f(z^{j}x z^{-k}) =
\sum_{j,k=0}^{N} \theta(k)\overline{\theta(j)} f(z^{-k+j}x)\notag
\\& =\sum_{n=-N}^N \sum_{k=\max(n,0)}^{\min(N,N+n)} \theta(k)\overline{\theta(k-n)} f(z^{-n} x) = \sum_{n=-N}^N \psi(n) f(z^{-n} x) = \sum_{n\in S} \psi(n) f(z^{-n} x),
\end{align}
for $n\not\in S$ provides $\psi(n)=0$, hence the term drops
out. Thus we can describe the support of $F$ better: $\supp F \Subset \bigcup_{n\in S} \supp f(z^{-n}x) = \bigcup_{n\in S} ~z^{n} \supp f \Subset \bigcup_{n\in S} U_{n}^*
\Subset \Om$ in view of (ii).

Therefore, with $F_0:=\frac{1}{F(e)} F$, we obtain $F_0\in
\FF^c(\Om)$, noting that $F(e)\ne 0$, as $F$ is not identically
zero and $F\gg 0$. Actually, we can compute the precise values
of $F(e)$ and $F(z)$ easily: $F(e)=\sum_{n\in S} \psi(n)
f(z^{-n})=\psi(0)f(e)=f(e)$ as $z^n\notin U$ for
$0<|n|\leq N$ by (iii), and $F(z)=\sum_{n\in S} \psi(n)
f(z^{-n+1})=\psi(1)f(e)$, for the same reason.
Finally, $F_0(z)=\frac{1}{F(e)} F(z)=\frac{1}{f(e)}f(e) {\psi(1)}={\psi(1)}$, whence $ |\psi(1)|=|F_0(z)| \leq \sup \{|\phi(z)|~:~ \phi \in \FF^c_G(\Om)\}=\CF^c_G(\Om,z)$.
Taking supremum on the left hand side over all $\psi\in
\FF^c_\ZZ(H)$ thus leads to $\CF^c(H)\leq \CF^c_G(\Om,z)$.
\end{proof}

\begin{proof}[Proof of Theorem \ref{th:KrenciStrongfini}]
The proof is rather similar to the preceding one, once we carefully change all references from $\ZZ$ to $\ZZ_m$, $\CF(H)$ to $\CF_m(H)$ and $\FF_\ZZ^{\#}(\Om,z)$ and $\FF_\ZZ^{c}(\Om,z)$ to $\FF_{\ZZ_m}(\Om,z)$, and note that $Z:=\langle z \rangle$ is now only a finite subgroup with $Z\cong \ZZ_m$, so the natural isomorphism $\eta(k):=z^k$ acts between $\ZZ_m$ and $Z$ now. Let us briefly work out the main steps of the argument.

Note that $H:=H(\Om,z)$ is finite  together with the group $\ZZ_m$, so of course there is no difference between the formulations with respect to continuity or compactness assumptions, and the notations $\CF_m(H)$ and $\KK_m(H)$ of \eqref{eq:CFOmZ} and \eqref{CFKKconstants} are in effect.

Again it suffices to prove $\CF^c_G(\Om,z)=\CF^{\#}_G(\Om,z)=\CF_m(H)$ and $\KK^c_G(\Om,z)=\KK^{\#}_G(\Om,z)=\KK_m(H)$. As before, the complex and real cases are entirely similar, so again we work out only the complex case. Once again, as $\CF^c_G(\Om,z)\leq \CF^{\#}_G(\Om,z)$, we are to prove only two inequalities, the first being that $\CF^{\#}_G(\Om,z) \leq \CF_m(H)$.

Now take any $f\in \FF^{\#}_G(\Om)$. In view of Lemma \ref{l:furtherposdeffunctions} {\it{(i)}}, $g:=f|_Z \gg 0$ on $Z$, thus we have defined a function $g \in \FF_Z((\Om\cap Z))$. Note that as $o(z)=m$, we have $Z\simeq \ZZ_m$, and in particular finite. The natural isomorphism $\eta:\ZZ_m\to Z$ carries over $g$, defined on $Z\le G$, to a positive definite function $\psi:=g\circ \eta: \ZZ_m\to \CC$ on $\ZZ_m$, which has normalized value $\psi(0)=g(e)=f(e)=1$, and satisfies $\supp \psi \subset H$ for $\supp g \subset (\supp f \cap Z) \subset (\Om \cap Z) = \{ z^k~:~ z^k\in \Om\}=\eta(H)$ by definition of $H:=H(\Om,z)$ and $\eta:\ZZ_m\to Z$.

From here we read that $|f(z)|=|g(z)|=|\psi(1)| \leq \sup \left\{ |\f(1)| ~:~  \f \in \FF_{\ZZ_m}(H) \right\} =\CF_m(H)$. Taking $\sup_{f\in \FF^{\#}_G(\Om)}$ on the left hand side concludes the proof of the first part.

\medskip

It remains to show $\CF_m(H) \leq \CF^{c}_G(\Om,z)$, so the proof hinges upon the construction to an arbitrary function $\psi\in \FF_{\ZZ_m}(H)$ another function $F_0\in \FF^{c}_G(\Om,z)$ with $|\psi(1)|\leq |F_0(z)|$. Denote $S:=\supp \psi \subset \ZZ_m$.

As before, we manipulate using the convolution square-root of $\psi$, provided now by Lemma \ref{l:FRfinite}. We thus get the sequence $\theta:\ZZ_m \to \CC$ with $\psi(n)=\sum_{k\mod m} \theta(k)\overline{\theta(k-n)} =(\theta\star\widetilde{\theta}) (n)$.

However, now without any further condition we can always find for any given neighborhood $U$ of $e$ a function $f\gg 0$, continuous and supported compactly in $U$, with the invariance property with respect to conjugation by $z$. Indeed, according to the above Proposition \ref{p:invariance}, such a function exists for any neighborhood $U$ of $e$ if and only if $z\in G$ is a round element: however, that holds true in view of Example \ref{ex:ozm}. So we have such a function $f$.

At this point, as in the previous proof for Theorem \ref{th:KrenciCFstrong}, we chose a proper neighborhood $U$ of $e$, satisfying similarly to (i), (ii), and (iii) there the following analogous properties.
\begin{itemize}
\item[(i)] $U^*:=\bigcup_{k\in\ZZ_m} z^{-k}Uz^k \subset \Om$;
\item[(ii)] $U_n^*:=z^n U^* \subset \Om$ whenever $n\in S$;
\item[(iii)] If $0\ne n \in \ZZ_m$, then $z^n\not\in U^*$, i.e.
    $e\not\in z^{-n}U^*=U_{-n}^*$; still equivalently,
    $z^n\not\in U^*_k$ unless $n=k \mod m$.
\end{itemize}

Choosing $U$ a sufficiently small neighborhood of $e\in\Om$, for any $k\in \ZZ_m$ together with $U$ also the homeomorphic image $z^{-k}Uz^k$ of $U$ will stay within $\Om$, hence (i) holds for $U^{*}$.

Now if $n\in S\subset H$, then also $z^n\in \Om$, i.e. $e\in
z^{-n} \Om$, and for fixed $n$ a suitably small choice of $U$
can ascertain that the $m$ homomorphic images $z^{-k}Uz^k$ of
$U$ all stay within $z^{-n}\Om$ -- that is, $z^n U^* \subset
\Om$. Considering the necessary restrictions for the finitely
many elements $n\in S$ it is also possible to simultaneously
satisfy (ii) for all $n\in S$.

Finally, consider (iii). Let us take $\Om^*:=\Om\setminus
\{z^n~:~ 0\ne n \mod m\}$, which is a finite modification of
$\Om$, and is still a symmetric open neighborhood of $e$ (for
$z^n\ne e$ if $0\ne n \mod m$). So arguing as in (i), we can thus
ascertain that $U^*\subset \Om^*$, whence $z^n\not\in U^*$, and (iii) holds.

\smallskip

Finally, we define the new function $F(x):=\sum_{j\mod m} \sum_{k\mod m} \theta(k)\overline{\theta(j)} f(z^{j}xz^{-k})$, with the above $f$ as before. Then exactly as above, we find $F\gg 0$, $F\in C(G)$, $\supp F \subset \Om$, $F(e)=f(e)$ and $F(z)={\psi(1)} f(e)$, thus $F_0:=\frac{1}{F(e)} F \in \FF^{c}_G(\Om,z)$ and $|F_0(z)|=|\psi(1)|$, concluding the argument also in this case.
\end{proof}

\end{document}